\documentclass[12pt]{amsart}

\usepackage{latexsym, amsmath, amscd, amssymb, amsthm} 
\usepackage[T2A]{fontenc}
 \usepackage[cp1251]{inputenc}
\newtheorem{te}{Theorem}
\newtheorem{lm}{Lemma}

\begin{document}

\noindent

 \title{  A  note on the  Poincar\'e series   of the invariants  of   ternary  forms  }

\author{Leonid Bedratyuk}\address{Khmelnytsky National University, Instytuts'ka st. 11, Khmelnytsky, 29016, Ukraine}

\begin{abstract}

Analogue of Springer's formula for the  Poincar\'e series of  the  algebra invariants of  ternary  form  is  found.
\end{abstract}
\maketitle

\noindent
{
\noindent
{\bf 1.}  Let  $I_{n,d}$ be graded algebra of invariants  of  $n$-ary form of degree  $d$:

$$
I_{n,d}=(I_{n,d})_0+(I_{n,d})_1+\cdots+(I_{n,d})_k+ \cdots,
$$
and 
$$
P_{n,d}(T)=1+\dim((I_{n,d})_1) T+\cdots +\dim((I_{n,d})_k) T^k+\cdots
$$ 
be its  Poincar\'e series. By using Sylvester-Cayley formula,  the series  $P_{2,d}(T)$  for  small   $d$  were calculated by Sylvester and Franclin in     \cite{SF}. The explicit  formula for calculation of  the Poincare  series was derived by Springer, see \cite{SP}. Namely,
$$
P_{2,d}(z)=\sum_{0\leq k <d/2} (-1)^k \psi_{d-2\,k} \left( \frac{(1-z^2)z^{k(k+1)}}{(1-z^2)(1-z^4)\ldots(1-z^{2k}) (1-z^2)(1-z^4)\ldots (1-z^{2d-2j})} \right),
$$
here
$$
(\psi_n f)(z^n)=\frac{1}{n} \sum_{k=1}^n f(e^{\frac{2i\pi k}{n}}z),
$$
for arbitrary rational function  $f \in \mathbb{C}(t).$  
By  using  the  formula the Poicare series $P_{2,d}(T)$   for  $d<17$  was calculated  in \cite{BC}. Also,  by using a  variation of the  formula  in   \cite{LP}   an algorithm for the computation of the Poincar\'e series is proposed  and these   series were calculated for even $d\leq 36.$ 
 
 For the  Poincare series of  compact group   $G$ there exists  the Molien–Weyl integral formula. In  the  case 
 $G=SL(2, \mathbb{C})$ it  can  be written in  the following  form
$$
P_{2,d}(T)=\frac{1}{2 \pi i} \oint_{|z|=1} \frac{1-z^{-2}}{\prod_{k=1}^n (1-T z^{d-2\,k})} \frac{dz}{z},   |z|<1,
$$
 see  \cite{DerK},  page 183.   By  calculation   the  integral, in  \cite{DD} the   series  $P_{2,d}(T)$   was found  for $d\leq 30.$

We  know very  little about the Poicare series $P_{3,d}(T)$  for  the algebra  of  invariants of  ternary  form.  In  the  paper     \cite{Shi}  the  series  $P_{3,\,4}(T)$   was  calculated. Also, in  the paper \cite{ASM}  an  analogue of Sylvester-Cayley  formula was  derived and listed several  first terms of  Poincare series $P_{3,d}(t)$  for small $d\leq 7.$ 

The aim of this paper is to derive an analogue of  Springer's formula for the  Poincare series of  the  algebra invariants of the ternary  form.

{\bf 2.}  We begin by short proving the Springer's formula for the Poincar\'e  series of the  algebra invariants of  the binary form.  Let us  consider the $\mathbb{C}$-algebra $\mathbb{C}[[z]]$   of  formal   power series in    $z.$
For arbitrary   $n \in \mathbb{N}$  define  $\mathbb{C}$-linear function
$$\varphi_{n}: \mathbb{C}[[z]] \to \mathbb{C}[[z]]$$  in the  following  way
$$
\varphi_{n}\bigl(z^m\bigr):=\left \{ \begin{array}{l} z^{\frac{m}{n}}, \text{ if  } m= 0 \mod n,  \\ 0, \text{  if } m \neq 0 \mod n,\\ 1,   \text {    for }   m=0. \end{array} \right. 
$$

Then  for arbitrary series 

 $$A=a_0+a_1 z+a_2 z^2+ \cdots ,$$ we  get 
$$
\varphi_{n}(A)=a_0+a_n z+a_{2\,n} z^2+\cdots+a_{s\, n} z^s+\cdots .
$$
The lemma give us the explicit forms of  the function 
 $\varphi_{n},$ $n>0$ 
\begin{lm} For any  $f \in \mathbb{C}[[z]]$    the following representations  hold
$$
\begin{array}{ll}
 (i) &  \displaystyle \varphi_{n}(f(z))=\frac{1}{n} \sum_{k=1}^{n} f(z \,e^{\frac{2\pi i\,k}{n}} ) \Big|_{z^n=z};
 \\
 (ii) &  \displaystyle \varphi_{n}(f(z))=\frac{1}{2 \pi} \int_{0}^{2\pi} \frac{f(z e^{i \theta})}{1-e^{-i\, n \,\theta}} d\theta \Big|_{z^n=z}.
 \\
%   (iii)  & \displaystyle 
%\varphi_{n}(f(z))=\frac{1}{2 \pi i} \oint_{C_n} \frac{f(z r)}{(r^n-1)} \frac{dr}{r}\Big|_{z^n=T}.
\end{array}
$$
%Інтегрування  ведеться   по контуру $C_n$ який  охоплює  всі  корені  з $1$ степеня  $n$   і  не  містить  0. 

\end{lm}
\begin{proof}
\noindent
$(i)$

Put  $\displaystyle \xi= e^{\frac{2\pi i}{n}}.$ We  have  
$$
\frac{1}{n}\sum_{k=1}^n \bigl(\xi^k z \bigr)^m=z^m \frac{1}{n} \sum_{k=1}^n \bigl(\xi^k \bigr)^m=\left \{ \begin{array}{l} \displaystyle z^m, \text{  if  } {m=0 \pmod n},  \\  \\ 0, \text{ otherwise. } \end{array} \right.
$$
It  is follows that 
$$
 \varphi_{n}(f(z))(z^m)=\left \{ \begin{array}{l} \displaystyle z^{\frac{m}{n}} \text{  if  } {m=0 \pmod n},  \\  \\ 0, \text{  otherwise } \end{array} \right.
$$
\noindent

This construction is  due  to Simson, see \cite{TS}  or   \cite{ASF},  page  14.

\noindent
$(ii)$   Set  $f(z)=\sum_{k=0}^{\infty} f_k z^k.$ Then,  taking into account that for integer  $n$  the integral   $\displaystyle \int_0^{2\pi} e^{i\, n \theta}   d\theta$ is equal to  zero,  we  get
$$
\frac{1}{2 \pi} \int_{0}^{2\pi} \frac{f(z e^{i \theta})}{1-e^{-i\, n \,\theta}} d\theta=\frac{1}{2 \pi} \int_{0}^{2\pi} \sum_{k,s=0}^{\infty} f_k z^k e^{k\,i \theta} e^{-s n \theta}=\frac{1}{2 \pi}  \sum_{k=0}^{\infty} f_k  z^k \sum_{s=0}^{\infty}  \int_{0}^{2\pi} e^{(k-s\,n) i \theta} d\theta=
$$
$$
=\frac{1}{2 \pi}  \int_{0}^{2\pi} \sum_{s=0}^{\infty} f_{n\,s} z^{n\,s} d\theta=  \sum_{s=0}^{\infty} f_{n\,s} z^{n\,s}.
$$
After replasing  $z^n$ by $z$ we  obtain the statement of  the lemma.

%$(iii)$ В інтегралі  $   \displaystyle \frac{1}{2 \pi} \int_{0}^{2\pi} \frac{f(z e^{i \theta})}{1-e^{-i\, n \,\theta}}\, d\theta $  виконаємо  
%заміну змінної $r=e^{i \theta}.$  Тоді, враховуючи   аналітичність $f(z),$    отримаємо
%$$
% \frac{1}{2 \pi} \int_{0}^{2\pi} \frac{f(z e^{i \theta})}{1-e^{-i\, n \,\theta}}\, d\theta= \frac{1}{2 \pi i} \int_{|r|=1} \frac{r^{n-1}f(z r)}{r^n-1}\, 
%dr=\sum_{k=1}^n {\rm Res}\Bigl( \frac{r^{n-1}f(z r)}{r^n-1},r=\xi^k \Bigr)=
%$$
%$$
%=\sum_{k=1}^n {\rm Res}\Bigl( \frac{f(z r)}{r(r^n-1)},r=\xi^k \Bigr)=\frac{1}{2 \pi i} \oint_{C_n} \frac{f(z r)}{(r^n-1)} \frac{dr}{r}.
%$$

\end{proof}

As  above,  to work with formal  power  series in  two letters, define  $\mathbb{C}$-linear function $ \Psi_{m,n}:\mathbb{C}[[t,z]] \to \mathbb{C}[[T]],$ $ m,n  \in \mathbb{N} $ by 
$$
\Psi_{n_1,n_2}\bigl(t^{m_1} z^{m_2}\bigr)=\left \{ \begin{array}{l} T^{s}, \text{  if  } \displaystyle  \frac{m_1}{n_1}=\frac{m_2}{n_2}= s \in \mathbb{N}, \\  \\ 0, \text{ otherwise. } \end{array} \right. 
$$
If  now
$$
A=a_{0,0}+a_{1,0} t +a_{0,1}z+ a_{2,0}t^2+\cdots,
$$
then, obviously,
$$
\Psi_{n_1,n_2}(A)=a_{0,0}+a_{n_1,n_2} T+a_{2\,n_1,2\,n_2}T^2+\cdots  .
$$
In  some  cases an  calculation the functions $ \Psi_{m,n}$  can  be reduced   to  calculation of  the functions  such as   $\varphi$.  The following  statement holds: 
\begin{lm}
%$$
%\begin{array}{ll}
% (i) & \displaystyle \Psi_{m,n} =\Psi_{1,1} \circ \varphi_m \circ \varphi_n,  \text{  тут } \varphi_n: \mathbb{C}[[t,z]] \to \mathbb
%{C}%[[z,T_1]],   \varphi_m: \mathbb{C}[[z,T_1]] \to \mathbb{C}[[T_1,T_2]], \\ \\
 %& \Psi_{m,n}: \mathbb{C}[[T_1,T_2]] \to \mathbb{C}[[T]]; \\

%& \\
 For any    $R \in \mathbb{C}[[z]]$      and  for    $m, n, k  \in \mathbb{N}$    we  have 
$$
 \Psi_{m,n}\left( \frac{R}{1-t^m z^k} \right)=\left \{ \begin{array}{l} \varphi_{n-k}(R)  \text{ if  } n\geq k, \\  \\ 0,  \text{  if  } n < k \end{array} \right. 
$$
% (iii) & \displaystyle \Psi_{1,n}\left( f(t,z) \right)=\left(\frac{1}{ (2 \pi i)} \oint_{|r|=1} f(t/r^n,z r) \frac{dr}{r} \right)  \Big | _{t z^n=T}.

% (iv)  &  \displaystyle \Psi_{m,n}\left( f(t,z) \right)=\left(\frac{1}{ (2 \pi i)^2} \oint_{|r_2|=1}  \oint_{|r_1|=1} \frac{ r_1^{n-1} r_2^{m-1}f(t r_%1,z r_2)}{r_1^n r_2^m-1}dr_1 dr_2 \right) \Big | _{t^m z^n=T}.
%\end{array}
\end{lm}
\begin{proof}

\noindent
%$(i)$ Нехай $f=\sum_{k,s=0}^{\infty} f_{k,s}t^k z^s.$  Тоді,  за означенням, $ \Psi_{m,n}(f)=\sum_{k=0}^{\infty}f_{k\,m, k\,n} ^k.$
%З    іншого боку,  послідовно  обчислюючи,  отримаємо
 %$$
%\begin{array}{l}
%\displaystyle \varphi_n(f) =\sum_{k,s=0}^{\infty}f_{k\,n, s} T_1^k z^s,\\
%\displaystyle  {\varphi_m(\varphi_n(f)) =\sum_{k,s=0}^{\infty}f_{k\,n, s\,m} T_1^k T_2^s,}\\
%\displaystyle  \Psi_{1,1}(\varphi_m(\varphi_n(f)))=\sum_{k=0}^{\infty}f_{k\,m, k\,n} T^k.
%\end{array}
%$$
%Отже, $\Psi_{m,n} =\Psi_{1,1} \circ \varphi_m \circ \varphi_n.$
%\noindent
 Put  $R=\sum_{j=0}^{\infty} f_{j} z^j,$ $f_j \in \mathbb{C}.$  Then for   $k < n$  we  get 
$$
 \Psi_{m,n}\left( \frac{R}{1-t^m z^k} \right)=\Psi_{m,n}\Big( \sum_{j,s\geq 0} f_j  z^j (t^m z^k)^s\Big)  =\Psi_{m,n}\Big(\sum_{s\geq 0} f_{s(n-k)} (t^m z^n)^s \Big)=\sum_{s \geq 0}  f_{s(n-k)} T^s.
$$
On the other hand  $\displaystyle \varphi_{n-k}(R)=\varphi_{n-k}\Bigl(\sum_{j=0}^{\infty} f_{j} z^j\Bigr)=\sum_{s \geq 0}  f_{s(n-k)} T^s.$
 \end{proof}
As in the proof of Lemma 1   we obtain that
\noindent

$$
 \Psi_{1,1}\left( f(t,z) \right)=\frac{1}{2\pi } \int_0^{2 \pi} f(t\,e^{- i \theta}, z\,e^{ i \theta})\,  d\theta. 
$$

The main idea of calculations of this work is that 
the  Poincar\'e series $ P_ {2, d} (T) $ can be expressed  in terms of functions $ \Psi.$ The following simple but important statement  holds

\begin{lm} 
$$
P_{2,d}(T)=\Psi_{1,d}\bigl(f_d(t,z^2)\bigr).
$$
where
$$
f_d(t,z)=\frac{(1-z)}{(1-t)(1-t z)\ldots (1-t z^d)}=\frac{1-z}{(t,z)_{d+1}}.
$$
and  $(a,q)_n=(1-a) (1-a\,q)\cdots (1-a\,q^{n-1})$ -- $q$-shifted factorial.

\end{lm}
\begin{proof}  For any  $A \in \mathbb{C}[[t,z]]$      denote  by  $[t^n z^m]A$  the coefficients in  $t^n z^m.$
The Sylvester-Cayley formula   implies  that  the dimension  of  the vector space  $(I_{2,d})_n$ is  equal  to  $[(t z^{d/2})^n]f_d(t,z).$ 
Then
$$
P_{2,d}(T) = \sum_{n=0}^{\infty}  \dim(I_{2,d})_n T^n= \sum_{n=0}^{\infty} \bigl([(t z^{d/2})^n]f_d(t,z)\bigr) T^n =\sum_{n=0}^{\infty} \bigl([(t z^{d})^n]f_d(t,z^2)\bigr)T^n=
$$
$$
= \Psi_{1,d}(f_d(t,z^2)).
$$
\end{proof}

%Як наслідок,   отримаємо ще одну формулу для  ряду  Пуанкаре   алгебри  інваріантів бінарної  форми
%$$
%P_{2,d}(T) =\frac{1}{2\pi i} \oint_{|r|=1} \frac{ 1-z^2 r^2}{ \prod_{k=0}^{d}  (1-t z^{2 k} r^{d-2\,k})} \frac{dr}{r}.
%$$
Now we  can present  simple proof of  the Springer formula  for the Poincar\'e  series  $P_{2,d}(T)$
\begin{te}[Springer]
$$
P_{2,d}(T)=\sum_{0\leq k <d/2} \varphi_{d-2\,k} \left( \frac{(-1)^k z^{k(k+1)} (1-z^2)}{(z^2,z^2)_k\,(z^2,z^2)_{d-k}} \right),
$$
\end{te}
\begin{proof}
Consider the partial fraction decomposition of  the rational function $f_d(t,z):$ 
$$
f_d(t,z)=\sum_{k=0}^{d} \frac{R_k(z)}{1-t z^k}.
$$
It  is easy  to  see,  that  
$$
\begin{array}{l}
\displaystyle  R_k(z)=\lim_{t \to z^{-k}}\left( f_d(t,z)(1-t z^k) \right)=\\
\displaystyle =\frac{1-z}{(1-z^{-k}) \ldots (1-z^{-k} z^{k-1}) \cdots (1-z^{-k} z^{k+1}) \cdots (1-z^{-k} z^{d})}=
\\
\displaystyle=\frac{1-z}{z^{-(k+(k-1)+\ldots+1)}(z^{k}-1)(z^{k-1}-1) \ldots (z-1) \cdots (1-z) \cdots (1-z^{d-k})}=
\\
\\
\displaystyle=\frac{(-1)^k z^{\frac{k(k+1)}{2}} (1-z)}{(z,z)_{k} (z,z)_{d-k}}.
\end{array}
$$
Using  the above lemmas we obtain
$$
P_{2,d}=\Psi_{1,d}\bigl(f_d(t,z^2)\bigr)=\Psi_{1,d}\left( \sum_{k=0}^{n} \frac{R_k(z^2)}{1-t z^{2\,k}} \right)= \sum_{0\leq k <d/2} \varphi_{d-2\,k} \left( \frac{(-1)^k z^{k(k+1)}(1-z^2)}{(z^2,z^2)_k (z^2,z^2)_{d-k}} \right).
$$
\end{proof}

%%%%%%%%%%%%%%%%%%%%%%%%%%%%%%

{\bf 3.}  Let us derive an formula for the Poicar\'e series  of  algebra of  invariants of  ternary form.  In  \cite{ASM}  was  proved that the  dimension  of  the vector  space    $(I_{3,d})_n $  is  equal  to   $\displaystyle [ \bigl(t (pq)^{\frac{d}{3}}\bigr)^n] f_d(t,p,q) ,$  where

$$
f_d(t,p,q)=\frac{b_3(p,q)}{\displaystyle \prod_{k+l \leq  d } (1-t p^k q^l)}=\frac{b_3(p,q)}{\displaystyle \prod_{s=0}^d (t q^s,p)_{d+1-s}},
$$
and $b_3(p,q)=1+p\,q+\displaystyle \frac{q^2}{p}-2\,q-q^2.$

On  the ring  $\mathbb{C}[[t,p,q,p^{-1},q^{-1}]]$ let  us  define  the    $\mathbb{C}$-linear  functions  $ \Phi_{n_1,n_2,n_3}$ and  $\hat \Phi_{n_1,n_2},$ $n_i\geq 0$  in the  following  way
$$
\Phi_{n_1,n_2,n_3}\bigl(t^{m_1} p^{m_2} q^{m_3}\bigr)=\left \{ \begin{array}{l}T^{s}, \text{ if  } \displaystyle  \frac{m_1}{n_1}=\frac{m_2}{n_2}= \frac{m_3}{n_3}=s \in \mathbb{N},  m_i \geq 0,  \\  \\ 0, \text{ otherwise. } \end{array} \right. 
$$
$$
\hat \Phi_{n_1,n_2}\bigl( p^{m_1} q^{m_2}\bigr)=\left \{ \begin{array}{l} T^{s}, \text{  if   } \displaystyle \frac{m_1}{n_1}= \frac{m_2}{n_2}=s \in \mathbb{N},  m_i \geq 0,  \\  \\ 0, \text{ otherwise } \end{array} \right. 
$$

The following result may be proved in much the same way as Lemma 2 and  Lemma 3 : 

\begin{lm}
$$
\begin{array}{ll}
(i) & \text{ For  arbitrary   } R \in \mathbb{C}[[p,q,p^{-1},q^{-1}]] \text{  and  for any natural  }   m, n, k  \\
& \text{ we  have:} \\
\\
& \displaystyle \Phi_{n_1,n_2,n_3}\left( \frac{R}{1-t^{n_1} p^k q^l } \right)=\left \{ \begin{array}{l} \hat  \Phi_{n_2-k, n_3-l}(R)  \text{ if  } n_2 \geq k \text{   and   } n_3 \geq l,\\  \\ 0,  \text{  if  } n_2 < k  \text{ or   } n_3 < l. \end{array} \right. ;\\

%(ii) &   \displaystyle \Phi_{m,n,k} =\Psi_{1,1,1} \circ \varphi_m \circ \varphi_n \circ \varphi_k,  \text{  тут } \varphi_k:\mathbbC}
%[[t,p,q]] \to \mathbb{C}[[t,p,T_3]],\\ 
%& \\
%& \varphi_n: \mathbb{C}[[t,p,T_3]] \to \mathbb{C}[[t,T_2,T_3]],   \varphi_m: \mathbb{C}[[z,T_2,T_3]] \to \mathbb{C}[[T_1,T_2,T_
%3]], \\ 
%& \Psi_{1,1,1}: \mathbb{C}[[T_1,T_2,T_3]] \to \mathbb{C}[[T]];\\
%&\\
(ii)& P_{3,d}(T)=\Psi_{1,d,d}\bigl(f_d(t,p^3,q^3)\bigr).
\end{array}
$$
\end{lm}

Now  we  are able to  prove the analogue of Springer's formula for the  Poincar\'e series of  the  algebra invariants of  ternary  form.
\begin{te}
$$
P_{3,d}(T)=\sum_{0\leq k,j \leq [d/3]}\hat  \Phi_{d-3\,k,d-3\,j} \left( \frac{(-1)^k p^{\frac{3k(k+1)}{2}} b_3(p^3,q^3)}{ \Bigl( \prod_{s=0, s \neq j}^d (q^{-3k},p^{3(s-j)})_{d+1-s} \Bigr)  (p^3,p^3)_k\,(p^3,p^3)_{d-(k+j)}} \right).
$$
\end{te}
\begin{proof}
Consider the partial fraction decomposition of  the rational function $f_d(t,p,q):$
$$
f_d(t,p,q)=\sum_{k+j \leq d} \frac{R_{k,j} (p,q)}{1-t p^k q^j}.
$$
We  have 
$$
\begin{array}{l}
\displaystyle R_{k,j}(p,q)=\lim_{t \to p^{-k} q^{-j}}\left( f_d(t,p,q)(1-t p^k q^j) \right)=\lim_{t \to p^{-k} q^{-j}}\left(\frac{b_3(p,q)(1-t p^k q^j)}{\Bigl(\prod_{ s \neq j }^d (t q^{s},p)_{d+1-s}\Bigr) (t q^j,p)_{d+1-j}}\right)=\\

= \displaystyle \frac{b_3(p,q)}{\Bigl(\prod_{ s \neq j }^d (p^{-k}q^{s-j},p)_{d+1-s}\Bigr) (1-p^{-k})\ldots(1-p^{-1})(1-p)\ldots (1-p^{d-j-k})}=\\

=\displaystyle \frac{(-1)^k p^{\frac{k(k+1)}{2}} b_3(p,q)}{\Bigl(\prod_{ s \neq j }^d (p^{-k}q^{s-j},p)_{d+1-s}\Bigr) (p,p)_k (p,p)_{d-k-j}}.
\end{array}
$$
The  Lemma 4  now  yields
$$
P_{3,d}(T)=\Phi_{1,d,d}\left(\sum_{k+j \leq d  } \frac{R_{k,j}(p^3,q^3)}{1-t p^{3k}q^{3j}} \right) =\sum_{k+j \leq d} \hat \Phi_{d-3k,d-3j}\Bigl(R_{k,j}(p^3,q^3) \Bigr)=
$$
$$
=\sum_{0\leq k,j \leq [d/3]} \hat \Phi_{d-3k,d-3j}\Bigl(R_{k,j}(p^3,q^3) \Bigr).
$$
\end{proof}

\end{document}